\documentclass[11pt]{article}
\usepackage{amssymb}
\usepackage{amsmath}
\usepackage{graphicx}

\begin{document}

\newtheorem{theorem}{Theorem}[section]
\newtheorem{tha}{Theorem}
\newtheorem{conjecture}[theorem]{Conjecture}
\newtheorem{corollary}[theorem]{Corollary}
\newtheorem{lemma}[theorem]{Lemma}
\newtheorem{claim}[theorem]{Claim}
\newtheorem{proposition}[theorem]{Proposition}
\newtheorem{construction}[theorem]{Construction}
\newtheorem{definition}[theorem]{Definition}
\newtheorem{question}[theorem]{Question}
\newtheorem{problem}[theorem]{Problem}
\newtheorem{remark}[theorem]{Remark}
\newtheorem{observation}[theorem]{Observation}

\newcommand{\ex}{{\mathrm{ex}}}

\newcommand{\EX}{{\mathrm{EX}}}

\newcommand{\AR}{{\mathrm{AR}}}

\def\endproofbox{\hskip 1.3em\hfill\rule{6pt}{6pt}}
\newenvironment{proof}%
{%
\noindent{\it Proof.}
}%
{%
 \quad\hfill\endproofbox\vspace*{2ex}
}
\def\qed{\hskip 1.3em\hfill\rule{6pt}{6pt}}

\newcommand{\norm}[1]{\lVert#1\rVert}
\def\ce#1{\lceil #1 \rceil}
\def\fl#1{\lfloor #1 \rfloor}
\def\lr{\longrightarrow}
\def\e{\epsilon}
\def\cB{{\cal B}}
\def\cC{{\cal C}}
\def\cF{{\cal F}}
\def\cG{{\cal G}}
\def\cH{{\cal H}}
\def\ck{{\cal K}}
\def\cI{{\cal I}}
\def\cJ{{\cal J}}
\def\cL{{\cal L}}
\def\cM{{\cal M}}
\def\cP{{\cal P}}
\def\cQ{{\cal Q}}
\def\cS{{\cal S}}
\def\cT{{\cal T}}
\def\imp{\Longrightarrow}
\def\1e{\frac{1}{\e}\log \frac{1}{\e}}
\def\ne{n^{\e}}
\def\rad{ {\rm \, rad}}
\def\equ{\Longleftrightarrow}
\def\pkl{\mathbb{P}^{(k)}_l}
\def\podd{\mathbb{P}^{(k)}_{2t+1}}
\def\peven{\mathbb{P}^{(k)}_{2t+2}}
\def\wt{\widetilde}
\def\wh{\widehat}

\def\b{\mathbb{B}}
\def\wb{\widetilde{\b}_n}

\def\band{[\frac{n}{2}-{2\sqrt{n\ln n}}, \, \frac{n}{2}+{2\sqrt{n\ln n}}]}
\def \gap{{2\sqrt{n\ln n}}}

\voffset=-0.5in

\title{Set families with a forbidden induced subposet}
\author{
Edward Boehnlein\thanks{Dept. of Mathematics, Miami University, Oxford, OH 45056, E-mail: boehnlel@muohio.edu.}
 \quad and \quad
Tao Jiang\thanks{Dept. of Mathematics, Miami University, Oxford,
OH 45056, USA. E-mail: jiangt@muohio.edu. }  }

\date{June 12, 2011}

\maketitle

\begin{abstract}
For each poset $H$ whose Hasse diagram is a tree of height $k$, we
show that the largest size of a family $\cF$ of
subsets of $[n]=\{1,\ldots, n\}$ not containing $H$ as an induced
subposet is asymptotic to
$(k-1){n\choose \fl{n/2}}$. This extends the result of Bukh
\cite{bukh}, which in turn generalizes several known results
including Sperner's theorem. \footnote{posted on arxiv on June 12, 2011.}
\end{abstract}


\medskip


\section{Introduction}

A poset $G=(S,\leq)$ is a set $S$ equipped with a partial ordering $\leq$. We say that a poset $G=(S,\leq)$ contains
another poset $H=(S',\leq')$ as a {\it subposet} and write $H\subseteq G$
if there exists an injection $f: S'\to S$ such that $\forall u,v\in S' $ if $u\leq' v$ then $f(u)\leq f(v)$.
We say that $G=(S,\leq)$ contains $H=(S',\leq')$ as an {\it induced subposet} and write $H\subseteq^* G$  if there exists
an injection $f:S'\to S$ such that $\forall u,v\in S' \, u\leq' v $ if and only if $f(u)\leq f(v)$.

Given a positive integer $n$, let $[n]=\{1,2,\ldots, n\}$. The Boolean lattice $\mathbb{B}_n$  of order $n$ is the poset
$(2^{[n]},\subseteq)$. Throughout this paper, we automatically equip any family $\cF\subseteq 2^{[n]}$ with
the containment $\subseteq$ relation and thus view $\cF$ as a subposet of $\b_n$. Given a positive integer $n$ and
a  poset $H$,
let $La(n,H)$ denote the largest size of a family $\cF\subseteq \mathbb{B}_n$ that does not contain $H$ as a subposet.
Let $La^*(n,H)$ denote
the largest size of a family $\cF\subseteq \mathbb{B}_n$ that does not contain $H$ as an induced subposet.
The study of these functions dates back
to Sperner's theorem \cite{sperner} which asserts that the largest size of an antichain in the Boolean lattice of order $n$
equals ${n\choose \fl{n/2}}$, with equality attained by taking the middle level of the boolean lattice.
If we use $P_2$ to denote a chain of two elements, then Sperner's theorem says that $La(n,P_2)=La^*(n,P_2)={n\choose
\fl{n/2}}$.  Erd\H{o}s \cite{erdos} extended Sperner's theorem to show that $La(n,P_k)$, where $P_k$ is the chain of $k$ elements,
is the sum of the $k-1$ middle binomial coefficients in $n$ (i.e. the sum of the sizes of the middle $k-1$ levels of $\mathbb{B}_n$).
Consequently, $\lim_{n\to \infty} \frac{La(n,P_k)}{{n\choose \fl{n/2}}}=k-1$.

A systematic study of $La(n,H)$ started a few years ago, and  a series of results on $La(n, H)$ were developed.
In most of these results  $H$ is a  poset whose Hasse diagram is a tree or is a height-$2$ poset, where the {\it height} of $H$ is the largest cardinality
of a chain in $H$. We give a brief recount of some of these results. Let $V_k$ denote the the height-$2$ poset that consists of $k+1$ elements
$A, B_1,\ldots, B_k$ where $\forall i\in [k],  A\leq B_i$.  We call $V_r$ the {\it $r$-fork}.
Improving  earlier results of Thanh \cite{thanh}, De Bonis and Katona \cite{BK-fork} showed
$ La(n,V_k)= {n\choose \fl{n/2}}(1+\frac{k-1}{n}+\Theta(\frac{1}{n^2}))$.
Let $B$ denote the Butterfly poset on four elements $A_1,A_2, B_1,B_2$ where $\forall i,j\in [2], A_i\leq B_j$. De Bonis, Katona, and
Swanepoel \cite{BKS} showed that $La(n,B)={n\choose \fl{n/2}}+{n\choose \fl{n/2}+1}$. More generally, for $r,s\geq 2$ let $K_{r,s}$ denote the two-level poset
consisting of elements $A_1,\ldots, A_r, B_1,\ldots, B_s$ where $\forall i\in [r], j\in [s], A_i\leq B_j$. De Bonis and Katona \cite{BK-fork} showed
that $La(n,K_{r,s})\sim 2{n\choose \fl{n/2}}$, as $n\to\infty$. Extending ealier results on tree-like posets, Griggs and Lu \cite{griggs-lu} showed that if $T$ is any height-$2$ poset whose Hasse diagram is a tree, then $La(n,T)\sim {n\choose \fl{n/2}}$.  Independently,  Bukh \cite{bukh} obtained the following  more general result.

\begin{theorem}{\bf (Bukh \cite{bukh})} \label {bukh}
If $H$ is a finite poset whose Hasse diagram is a tree of height $k\geq 2$, then
$$La(n,H)=(k-1){n\choose \fl{n/2}}(1+O(1/n)).$$
\end{theorem}

Note that Bukh's result generalizes (in a loose sense) all prior results on posets whose Hasse diagram is a tree.
Furthermore, it also implies De Bonis and Katona's result   that $La(n,K_{r,s})\leq 2{n\choose \fl{n/2}}(1+O(\frac{1}{n}))$
for the following reason. Consider the three level poset
$H$ that consists of elements $A_1,\ldots, A_r, B$, $C_1,\ldots, C_t$ where $\forall i\in [r], A_i \leq B$ and
$\forall j\in [t], B\leq C_j$. By transitivity  $\forall i\in [r], j\in [t]$, $A_i\leq C_j$, and so
$H$ contains $K_{r,s}$ as a subposet. So, $La(n,K_{r,s})\leq La(n,H)\leq 2{n\choose \fl{n/2}}(1+O(1/n))$.

In this paper, we are concerned with finding (or avoiding, depending on the perspective) induced subposets in $\b_n$.
Generally speaking, induced subposets are harder to force,
since we need to enforce noncontainment as well as containment among corresponding members.
For instance, for a family $\cF\subseteq \b_n$ to contain the $2$-fork $V_2$ as an induced subposet, we need to
find three members of $A,B,C$ of $\cF$ satisfying $A\subseteq B, A\subseteq C$,
$B\not\subseteq C$, and $C\not\subseteq B$.
By comparison, for $\cF$ to contain $V_2$ just a subposet, we only need to ensure the existence of
$A,B,C\in \cF$ satisfying $A\subseteq B, A\subseteq C$.

Since a family $\cF\subseteq \b_n$ that doesn't contain $H$ as a subposet certainly doesn't contain $H$ as
an induced suposet, we always have $La^*(n,H)\geq La(n,H)$. In general, the determination of $La^*(n,H)$
seems to be harder than the determination of $La(n,H)$. The only result on $La^*(n,H)$ that we are aware of is due to
Carroll and Katona \cite{CK} who showed that ${n\choose \fl{n/2}}(1+\frac{1}{n}+\Omega(\frac{1}{n^2}))
\leq La^*(n,V_2)\leq {n\choose \fl{n/2}}(1+\frac{2}{n}+O(\frac{1}{n^2}))$.

\medskip

In this paper, we extend Bukh's result to establish an induced version of his result.

\begin{theorem}\label{main}
Let $H$ be a finite poset whose Hasse diagram is a tree of height $k\geq 2$. Then
$$La^*(n,H)=(k-1){n\choose \fl{n/2}}(1+o(1)).$$
\end{theorem}

For a lower bound on $La^*(n,H)$, let $\cF$ consist of the middle $k-1$ levels of the Boolean lattice $B_n$.
Clearly $\cF$ doesn't contain $H$ (as an induced subposet) and $|\cF|={n\choose \fl{n/2}}(1-O(1/n))$.
So $La^*(n,H)\geq {n\choose \fl{n/2}}(1-O(1/n))$. The upper bound follows from the following more specific statement.

\begin{theorem}\label{upper}
Let $H$ be a finite poset whose Hasse diagram is a tree of height $k\geq 2$.  Let $\epsilon$ be a small positive real.
Let $n$ be a sufficiently large  positive
integer depending on $H$ and $\epsilon$. Let $\cF\subseteq \b_n$ be a family with $|\cF|\geq (k-1+\epsilon){n\choose \fl{n/2}}$.
Then $\cF$ contains $H$ as an induced subposet.
\end{theorem}
To prove theorem \ref{upper}, we first make a quick reduction. As mentioned in \cite{griggs-lu}, using Chernoff's inequality,
it is easy to show that the number of sets $F\in 2^{[n]}$ satisfying $||F|-\frac{n}{2}|>2\sqrt{n\ln n}$ is as most $O(\frac{1}{n^{3/2}}{n\choose \fl{n/2}})$.
Define
$$\wb=\{v\in \b_n: |v|\in
\band\}.$$
 By our  discussion above, there are only $o({n\choose \fl{n/2}})$ members of $\b_n$ that lie outside $\wb$. So to prove Theorem \ref{upper} it suffices to prove

\begin{theorem}\label{upper2}
Let $H$ be a finite poset whose Hasse diagram is a tree of height $k\geq 2$.  Let $\epsilon$ be a small positive real.
Let $n$ be a sufficiently large  positive
integer depending on $H$ and $\epsilon$. Let $\cF\subseteq \wb$ be a family with $|\cF|\geq (k-1+\epsilon){n\choose \fl{n/2}}$.
Then $\cF$ contains $H$ as an induced subposet.
\end{theorem}
For the rest of the paper, we prove Theorem \ref{upper2}.

\section{Preliminaries}

In this section, we recall some facts in \cite{bukh} which will be used in our main arguments.
Given a poset $H$, let $D(H)$ denote its Hasse diagram.
We call a poset $H$  {\it $k$-saturated} if every maximal chain is of length $k$. Thus, in particular, $H$ has height $k$.

\begin{lemma} \label {saturated} {\bf (\cite{bukh})}
If $H$ is a finite poset with $D(H)$ being a tree of height $k$, then $H$ is an induced subposet of some saturated
finite poset $\wt{H}$ with $D(\wt{H})$ being a tree of height $k$.
\end{lemma}

Due to Lemma \ref{saturated}, for the rest of the paper, we will assume that $H$ is  $k$-saturated.
Let $H$ be a poset and $x,y\in H$ where $x\leq y$. Define $[x,y]=\{z\in H: x\leq z\leq y\}$ and call it
an {\it interval}.
An interval in $H$ that is a chain is called a {\it chain interval}. The statement we give below is equivalent to
the original one in \cite{bukh}.

\begin{lemma}\label {interval-removal} {\bf (\cite{bukh})}
Let $k\geq 2$. Suppose $H$ is a $k$-saturated poset that is not a chain and $D(H)$ is a tree.
There exists $v\in H$, which is a leaf in $D(H)$, and  a chain  interval $I=[v,u]$ or $[u,v]$ of length $|I|\leq k$ containing
$v$ such that $D(P\setminus I')$ is a tree and the poset $P\setminus I'$ is $k$-saturated, where $I'=I-\{u\}$.
\end{lemma}

Fix a positive integer $k$. A $k$-chain in $\b_n$ is just a chain in $\b_n$ with $k$ distinct members.
 A {\it full chain} of a Boolean lattice $\b_m$ of order $m$ is just a chain of length $m+1$.
So it starts with the top element of the lattice and ends with bottom element of the lattice and contains a member of each cardinality between $0$ and $m$.
Let $\cF\subseteq \b_n$ be a family.
Given a $k$-chain  $Q=(F_1,\ldots, F_k)$, where $F_1\supset F_2\supset \ldots \supset F_k$ and $\forall i\in [k],  F_i\in \cF$, and a full chain $M$
of $\b_n$  that contains $Q$,  we call the pair $(M,Q)$ a {\it $k$-marked chain}
with markers in $\cF$. We call $M$ the {\it host} of the $k$-marked chain
$(M,Q)$ and say that $M$ {\it hosts} $(M,Q)$. Throughout our paper, the family $\cF$ is fixed. So, if we  omit the phrase ``with markers in $\cF$'', it should be understood that the markers (the $F_i$'s) are in $\cF$.
Note that if $M$ and $M'$ are two distinct full chains of $\b_n$ that contain $Q$, then
$(M,Q)$ and $(M',Q)$ are in fact considered to be two distinct $k$-marked chains in our definition.
The following lemma is a claim contained in the proof of Lemma 4  in \cite{bukh} (Lemma \ref{many-marked-chains} below), we paraphrase it slightly as follows.
Recall that ${x\choose k}$ is defined to be $0$ when $x<k$.

\begin{lemma} \label{marked-chain-count}
Let $k\geq 2$ be a positive integer. Let $\cF\subseteq \b_n$. Let $\cC(\b_n)$ denote the set of $n!$ full chains of $\b_n$.
For each $M\in \cC(\b_n)$, let $x(M)$ denote the number of members of $\cF$ contained in $M$.
Let $\cL$ denote the family of all the  $k$-marked chains with markers in $\cF$. Then $$|\cL|=\sum_{M\in \cC(\b_n)} {x(M)\choose k}.$$
\end{lemma}
\begin{proof}
Given any $M\in \cC(\b_n)$, $M$ hosts exactly ${x(M)\choose k}$ many $k$-marked chains with markers in $\cF$.
So there are altogether $\sum_{M\in \cC(\b_n) }{x(M)\choose k}$ many $k$-marked chains with markers in $\cF$.
\end{proof}

The following lemma is established in \cite{bukh}. We rephrase the proof slightly differently.

\begin{lemma} \label{many-marked-chains} {\bf (\cite{bukh})}
Let $\epsilon$ be a small positive real. Let $n$ be a sufficiently large positive integer.
Let $\cF\subseteq \b_n$.  Let $\cL$ denote the family of all the $k$-marked chains with markers in $\cF$.
If $|\cF|\geq (t-1+\epsilon){n\choose \fl{n/2}}$, then
 $$|\cL|\geq (\epsilon/k) k!.$$
\end{lemma}
\begin{proof}
For each $i$, let $C_i$ denote the number of full chains $M$ of $\b_n$ with $x(M)=i$.
Let $X$ be  the random variable
that counts the number of members of $\cF$ contained in a random full chain $M$ of $\b_n$.
For each member $F\in \cF$, the probability that $M$ contains $F$ is precisely $\frac{1}{{n\choose |F|}}$.
Hence $E(X)=\sum_{F\in \cF} \frac{1}{{n\choose |F|}}\geq |\cF|\cdot \frac{1}{{n\choose \fl{n/2}}}\geq k-1+\epsilon$.
On the other hand,  by a direct counting argument we have $E(X)=\sum_i iC_i/n!$. Thus, $\sum_i iC_i\geq (k-1+\epsilon)n!$.
Clearly, $\sum_{i=1}^{k-1} iC_i\leq (k-1)n!$. So, $\sum_{i=k}^n iC_i\geq \epsilon n!$. For $i\geq k$, we have ${i\choose k}\geq \frac{i}{k}$.
By Lemma \ref{marked-chain-count}, the number of $k$-marked chains with members in $\cF$ equals $\sum_i C_i{i\choose k}\geq \sum_{i=k}^n C_i(i/k)
=(1/k)\sum_{i=k}^n iC_i\geq (\epsilon/k) n!$.
\end{proof}


\section{Forbidden neighborhoods}

Recall that elements of $\b_n$ are subsets of $[n]$.
We refer to elements of $\b_n$ as {\it vertices} in the lattice. If $v$ is a vertex in $\b_n$, it is also understood to be the
subset of $[n]$ that it represents. The {\it cardinality} or {\it weight} of $v$, denoted by $|v|$, is the cardinality of
the subsets of $[n]$ that $v$ represents. Even though the partial ordering associated with $\b_n$ is the
containment $\subseteq$ relation, we will continue to denote it by $\leq$ in most cases.
If $u,v\in \b_n$ and $u\leq v$, we call $u$ a {\it descendant of $v$} and we call $v$ an {\it ancestor of
$u$}.
Given a vertex $v$ in $\b_n$, the {\it down-set} $D(v)$ of $v$ is defined to be
$$D(v)=\{u\in \b_n: u\leq v\}.$$
In other words, $D(v)$ is the set of all descendants of $v$.  Note that if $|v|=m$, then $D(v)$ forms
a Boolean lattice  $\b_m$ of order $m$.
The {\it up-set} $U(v)$ of $v$ is defined to be
$$U(v)=\{u\in \b_n: v\leq u\}.$$
In other words, $U(v)$ is the set of all ancestors of $v$.
Note that if $|v|=m$, then $U(v)$ forms a Boolean lattice $\b_{n-m}$ of order $n-m$.
If $S$ is a set of vertices in $\b_n$, we define
$$D(S)=\bigcup_{v\in S} D(v) \quad \mbox{and} \quad U(S)=\bigcup_{v\in S} U(v).$$

Given a vertex $v\in \wb$, a set $S\subseteq \wb$, $S\cap U(v)=\emptyset$, define
\begin{equation} \label{D-star-definition}
D^*(v,S)=[(D(v)\setminus\{v\}) \cap (U(S)\cup D(S))]\cap \wb.
\end{equation}
We call $D^*(v,S)$ the {\it forbidden neighborhood}  of $S$ {\it under} $v$ in $\wb$.
Given a vertex $v\in \wb$, a set $S\subseteq \wb$, $S\cap D(v)=\emptyset$, let
\begin{equation} \label{D-star-definition}
U^*(v,S)=[(U(v)\setminus\{v\})\cap (U(S)\cup D(S))]\cap \wb.
\end{equation}
We call $U^*(v,S)$ the {\it forbidden neighborhood} of $S$ {\it above}  $v$ in $\wb$.

The next two lemmas play an important role in our arguments.

\begin{lemma}\label{chain-into-lower-zone}
Let $s$ be a fixed positive integer. Let $n$ be a sufficiently large positive integer.
Let $v\in \wb, S\subseteq \wb$, where  $S\cap U(v)=\emptyset$ and $|S|=s$.
Let $M$ be a uniformly chosen random full chain of $D(v)$ (among all $|v|!$ full chains
of $D(v)$). We have
$$\Pr ( M\cap (D^*(v,S))\neq \emptyset)\leq \frac{27s\sqrt{n\ln n}}{n}.$$
\end{lemma}
\begin{proof}
For any vertex $w$ in $(D(v)\setminus \{v\})\cap \wb$, the probability that $M$ contains $w$ is
$\frac{1}{{|v|\choose |w|}}\leq \frac{1}{|v|}\leq \frac{1}{n/3}=\frac{3}{n}$. Since
$|(S\cap (D(v)\setminus \{v\}))\cap \wb|\leq s$,
\begin{equation}\label{S-itself}
\Pr\left (M\cap [(S\cap (D(v)\setminus \{v\}))\cap\wb]\neq \emptyset\right) \leq \frac{3s}{n}.
\end{equation}
Let $\ell=|v|-(\frac{n}{2}-2\sqrt{n\ln n})$. Since $v\in \wb$, $\ell\leq 4\sqrt{n\ln n}$.
To bound the probability that $M$ intersects $U=(U(S)\cap (D(v)\setminus\{v\}))\cap \wb=U(S)\cap (D(v)\setminus\{v\})$, we first bound the probability
that $M$ is disjoint from $U$.  Note that $U=U(S\cap D(v))\cap(D(v)\setminus\{v\})$ since only a descendant of $v$
may have an ancestor in $D(v)\setminus \{v\}$.
Suppose $S\cap D(v)=\{y_1,\ldots, y_p\}$, where $p\leq s$.
By our assumptions, $\forall i\in [p], y_i\leq v$
and $|v|-|y_i|\leq \ell$ (since $y_i\in \wb$). When we view $v,y_1,\ldots, y_p$ as sets we have $|\bigcap_{i=1}^p y_i|\geq |v|-p\ell$.
Being a full chain of $D(v)$, we may view $M$ as being obtained by starting with the set $v$ and successively removing
an element in it. For $M$ not to enter $U(\{y_1,\ldots, y_p\})\setminus \{v\}$, it suffices that the first element
removed from the set $v$ is in $\bigcap_{i=1}^p y_i$. So the probability that $M$ does not intersect $U$ is
at least $|\bigcap_{i=1}^p y_i|/|v|\geq 1-(p\ell/|v|)\geq 1-s\ell/|v|$. Therefore
\begin{equation}\label{avoid-upgraph}
\Pr(M\cap [(U(S)\cap (D(v)\setminus \{v\}))\cap \wb]\neq \emptyset)\leq \frac{s\ell}{|v|}\leq \frac{4s\sqrt{n\ln n}}{n/3}=\frac{12s\sqrt{n\ln n}}{n}.
\end{equation}
Next, we bound the probability that $M$ intersects $D=(D(S) \cap (D(v))\setminus\{v\}))\cap \wb$. Again, we first bound the
probability that $M$ is disjoint from $D$. Suppose $S=\{z_1,\ldots, z_s\}$. Since $S\cap U(v)=\emptyset$,
$\forall i\in [s]$, we have  $v\not\leq z_i$. So set $v$ has an element $u_i$ that is not in set $z_i$.
When we form $M$ by successively removing elements of  set $v$, as long as each of the first $\ell$ steps removes
an element outside $\{u_1,\ldots, u_s\}$, $M$ would not enter $D$. We have
\begin{eqnarray*}
\Pr( M\cap [(D(S)\cap (D(v)\setminus \{v\}))\cap \wb]=\emptyset)&\geq& \frac{(|v|-s)(|v|-s-1)\cdots (|v|-s-\ell+1)}{|v|(|v|-1)\cdots (|v|-\ell+1)}\\
&=&(1-\frac{s}{|v|})(1-\frac{s}{|v|-1})\cdots (1-\frac{s}{|v|-\ell+1})\\
&\geq& (1-\frac{s}{|v|-\ell+1})^\ell\\
&\geq & (1-\frac{s}{n/3})^\ell \quad \mbox{ (for large $n$)}\\
&\geq & 1-\frac{s\ell}{n/3} \quad \mbox{ (for large $n$)}
\end{eqnarray*}
Therefore,
\begin{equation} \label{avoid-downgraph}
\Pr( M\cap [(D(S)\cap (D(v)\setminus \{v\}))\cap \wb]\neq \emptyset) \leq \frac{s\ell}{n/3}=\frac{3s\ell}{n}\leq \frac{12s\sqrt{n\ln n}}{n}.
\end{equation}
Combining Equations (\ref{S-itself}), (\ref{avoid-upgraph}), and (\ref{avoid-downgraph}), we get
$$\Pr(M\cap D^*(v,S)\neq \emptyset)\leq \frac{27s\sqrt{n\ln n}}{n},$$
for large $n$.
\end{proof}

By a similar argument, we also have

\begin{lemma}\label{chain-into-upper-zone}
Let $s$ be a fixed positive integer. Let $n$ be a sufficiently large positive integer.
Let $v\in \wb, S\subseteq \wb$, where $S\cap D(v)=\emptyset$ and $|S|=s$. Let $M$ be a uniformly chosen random full chain of $U(v)$ (from all $(n-|v|)!$ full chains of $U(v)$).
We have
$$\Pr ( M\cap U^*(v,S)\neq \emptyset)\leq \frac{27s\sqrt{n\ln n}}{n}.$$
\end{lemma}

\section{$k$-marked chains and related notions}

In this section, as in the rest of the paper, chains are viewed from top to bottom, unless otherwise specified.
Let $H$ be a poset whose Hasse diagram is a tree of height $k$.  Let $h=|V(H)|$.
Let $\epsilon$ be a small positive real and $n$ be a sufficiently large positive integer $n$.
Let $\cF\subseteq \wb$ with $|\cF|\geq (k-1+\epsilon){n\choose \fl{n/2}}$.
Let $\cL$  be a family of $k$-marked chains with markers in $\cF$.   For each $v\in \wb$ and
$d\in [k]$, let
$$\cL(v,d)=\{(M,Q)\in \cL: \mbox{ the $d$-th member of $Q$ is $v$} \}.$$
Let $\cC(\b_n)$ denote the set of all $n!$ full chains of $\b_n$.
Next, we are going to define the notion of {\it bad}. This is defined relative to $h=|V(H)|$, which is fixed throughout
this section.
For each $d\in [k]$, we define a vertex $v\in \wb$ to be {\it $d$-lower-bad} relative to $\cL$ if  there exists
a set $S\subseteq \wb, S\cap U(v)=\emptyset, |S|\leq h$,  such that
$$\cL(v,d)\neq \emptyset  \mbox{ and } \forall (M,Q)\in \cL(v,d),   Q\cap D^*(v,S)\neq \emptyset.$$
We call $S$ a {\it $d$-lower-witness} of $v$ relative to $\cL$.
Similarly, we define a vertex $v\in\wb$ to be {\it $d$-upper-bad} relative to $\cL$ if  there exists
a set $T\subseteq\wb,  T\cap D(v)=\emptyset, |T|\leq h$,  such that
$$\cL(v,d)\neq \emptyset  \mbox{ and } \forall (M,Q)\in \cL(v,d),   Q\cap U^*(v,T)\neq \emptyset.$$
We call $T$ a {\it $d$-upper-witness} of $v$  relative to $\cL$. Let $d\in [k]$.
Let $v\in \wb$ and $M\in \cC(\b_n)$. We say that $v$ is {\it $d$-lower-bad } relative to $M$ and $\cL$
if $v$ is  $d$-lower-bad  relative to $\cL$ and
there exists at least one $Q$ such that $(M,Q)\in \cL(v,d)$.  We say that $v$ is {\it $d$-upper-bad} relative
to $M$ and $\cL$ if  $v$ is $d$-upper-bad relative to $\cL$ and there exists at least one $Q$
such that $(M,Q)\in \cL(v,d)$.
A $k$-marked chain $(M,Q)$ is {\it good} relative to $\cL$
 if $Q$ doesn't contain a vertex $v$ that is either $d$-lower-bad or $d$-upper-bad relative to $M$ and $\cL$
for any $d\in [k]$. The following proposition follows immediately from the definitions above.

\begin{proposition} \label{disjoint}
Let $(M,Q)$ be a member of $\cL$ that is good relative to $\cL$, and let $v\in Q$. Suppose $v$ is
the $d$-th vertex of $Q$. Then for any set $S$ of at most $h$ vertices of $\b_n$, where $S\cap U(v)=\emptyset$,
there exists a member $(M',Q')\in \cL(v,d)$ that  is disjoint from
$D^*(v,S)$. For any set $T$ of at most $h$ vertices, where $T\cap D(v)=\emptyset$,
there exists a member $(M'',Q'')\in \cL(v,d)$ that  is disjoint
from $U^*(v,T)$.
\end{proposition}
\begin{proof}
Note that $(M,Q)\in \cL(v,d)$. By our assumption, $v$ is not $d$-lower-bad or $d$-upper-bad relative to $\cL$;
otherwise $v$ would be either $d$-lower-bad or $d$-upper-bad relative to $M$ and $\cL$,
contradicting $(M,Q)$ being good relative to $\cL$.
So, there is no $d$-lower witness of $v$ or $d$-upper-witness of $v$ of size at most $h$ and the claim follows.
\end{proof}

Now, for each $d\in [k]$ and for each $v \in \wb$ that is $d$-lower-bad relative to $\cL$,
we fix a corresponding $d$-lower-witness
$S_{v,d}$ of $v$. For each $d\in [k]$ and each $v\in \wb$ that is $d$-upper-bad relative to $\cL$, we
fix a corresponding $d$-upper-witness $T_{v,d}$.
A chain $x_1>y_1>x_2>y_2>\ldots >x_p>y_p$ in $\b_n$ is called a {\it $d$-lower-bad string} if
 for each $i\in [p]$, $x_i$ is $d$-lower-bad relative to $\cL$ and $y_i\in D^*(x_i,S_{{x_i},d})$.
Similarly, a chain $x_1<y_1<x_2<y_2<\ldots < x_p<y_p$ in $\b_n$ is called a {\it $d$-upper-bad string}
if for each $i\in [p]$, $x_i$ is $d$-upper-bad relative to $\cL$ and $y_i\in U^*(x_i, T_{{x_i},d})$.

Given a sequence $J=(j_1, j_2, \ldots, j_q)$ of numbers in $[n]$, where either $j_1<j_2<\ldots< j_q$ or
$j_1> j_2>\ldots > j_q$, and a chain $C$ in $\b_n$, let $C[J]$ denote the subchain of $C$ consisting of
the $j_1$-th, $j_2$-th, $\ldots, j_q$-th members of $\cF$ on $C$ (counted from the top).
If $C$ contains fewer than $q$ members of $\cF$, then $C[J]$ is defined to be the empty chain.
If $J$ contains only one number $j$, then we write $C[j]$ for $C[\{j\}]$.

\begin{lemma} \label{exponential}
Let $d\in [k]$.
Let $p$ be a positive integer. Let $J$ be an increasing sequence of $2p$ numbers  in $[n]$.
Let $v\in \wb$.
Let $M$ be a uniformly chosen random full chain of $D(v)$. Then
$$\Pr(M[J] \mbox{ forms a $d$-lower-bad string})\leq  (\frac{27h\sqrt{n\ln n }}{n})^p.$$
\end{lemma}
\begin{proof}
Let $\gamma=\frac{27h\sqrt{n\ln n }}{n}$.
We use induction on $p$. For fixed $p$, we prove the statement for all $J$ with $2p$ numbers and all $v\in \wb$.
For the basis step, let $p=1$. Suppose $J=(j_1,j_2)$, where $j_1<j_2$.
Let $v\in \wb$ be given. Let $M$ be a uniformly chosen random full chain of $D(v)$.
We have
\begin{eqnarray*}
\Pr(M[J] \mbox{ is a $d$-lower-bad string})&\leq&\sum_{u\in D(v)} \Pr(M[j_1]=u)\cdot \Pr(M[j_2]\in D^*(u,S_{u,d})\mid M[j_1]=u)
\end{eqnarray*}

Fix any $u\in D(v)$. The set of full chains $M$ of $D(v)$
satisfying $M[j_1]=u$ are precisely those concatenations of full
chains of $I(v,u)$ (the sublattice consisting of all $x$
satisfying $v\geq x\geq u$) that contain exactly $j_1$ members of
$\cF$ and all full chains of $D(u)$.  So, $\Pr(M[j_2]\in
D^*(u,S_{u,d})\mid M[j_1]=u)$ is the same as the probability that on a
uniformly chosen random full chain $M'$ of $D(u)$  the
$(j_2-j_1+1)$-th member of $\cF$ is in $D^*(u,S_{u,d})$. This
probability is certainly no more than the probability that $M'$
intersects $D^*(u,S_{u,d})$, which by Lemma
\ref{chain-into-lower-zone}, is at most $\gamma$. Hence,
\begin{eqnarray*}
\Pr(M[J] \mbox{ is a $d$-lower-bad string})&\leq&\sum_{u\in D(v)} [\Pr(M[j_1]=u)\cdot \gamma]\\
&=&\gamma \cdot \sum_{u\in D(v)} \Pr(M[j_1]=u)\leq \gamma,
\end{eqnarray*}
where the last inequality uses the fact that for different $u$ the events $M[j_1]=u$ are certainly
disjoint.
This proves the basis step. For the induction step, assume $p\geq 2$. Suppose the claim has been proved
for all $J'$ and $v\in \wb$, where $J'$ is an increasing sequence of $2p-2$ numbers.
Given a full chain $M$ of $D(v)$ and a vertex $y$ on $M$, we let $M_y$ denote the portion of $M$
from $y$ down. Let $J'=(j_3-j_2+1, j_4-j_2+1,\ldots, j_{2p}-j_2+1)$. We have

\begin{eqnarray*}
\Pr(M[J] \mbox{ is a $d$-lower-bad string})&\leq&\sum_{u\in D(v)} \sum_{y\in D^*(u,S_{u,d})}
[\Pr(M[j_1]=u)\cdot \Pr(M[j_2]=y\mid M[j_1]=u)\\
&& \cdot \Pr(M_y[J']) \mbox{ is a $d$-lower-bad string}\mid  M[j_1]=u, M[j_2]=y)]
\end{eqnarray*}

Using reasoning as in the basis step, given $M[j_1]=u, M[j_2]=y$, all full chains of $D(y)$ are
equally likely for $M_y$. So given $M[j_1]=u, M[j_2]=y$, the probability that $M_y[J']$  is a $d$-lower-bad string
is the same as the probability that given a random full chain $M'$ of $D(y)$, $M'[J']$
forms a $d$-lower-bad string. By induction hypothesis, this is at most $\gamma^{p-1}$. So,

\begin{eqnarray*}
\Pr(M[J] \mbox{ is a $d$-lower-bad string})&\leq&\sum_{u\in D(v)} \sum_{y\in D^*(u,S_{u,d})}
[\Pr(M[j_1]=u)\cdot \Pr(M[j_2]=y\mid M[j_1]=u)\cdot \gamma^{p-1}]\\
&=& \gamma^{p-1}\cdot \sum_{u\in D(v)} \Pr(M[j_1]=u)\cdot \sum_{y\in D^*(u,S_{u,d})} \Pr(M[j_2]=y\mid M[j_1]=u)\\
&\leq& \gamma^{p-1}\cdot \sum_{u\in D(v)} \Pr(M[j_1]=u)\cdot \gamma \quad \mbox{(see discussion in the basis step)}\\
&=& \gamma^p\cdot \sum_{u\in D(v)} \Pr(M[j_1]=u)\leq \gamma^p.
\end{eqnarray*}
This completes the induction step and our proof.
\end{proof}

Using a similar argument, we have

\begin{lemma} \label{exponential2}
Let $d\in [k]$.
Let $p$ be a positive integer. Let $J$ be a decreasing sequence of $2p$ numbers  in $[n]$.
Let $v\in \wb$.
Let $M$ be a uniformly chosen random full chain of $U(v)$. Then
$$\Pr(M[J] \mbox{ forms a $d$-upper-bad string})\leq  (\frac{27h\sqrt{n\ln n }}{n})^p.$$
\end{lemma}


\section{A nested sequence of dense families of $k$-marked chains}

We show in this section that we can obtain a sequence of families of $k$-marked chains with markers in $\cF$,
$\cL_1\supseteq \cL_2\ldots \supseteq \cL_h$, such that for each $i\in [h]$, $|\cL_i|\geq  (\epsilon n!/k) (1-\frac{i}{2k})$
and for each $i\in [h-1]$  every member of $\cL_{i+1}$ is good relative to $\cL_i$.
Let $\cC(\b_n)$ denote the set of full chains of $\b_n$.

\begin{theorem} \label{nested}
Let $k, h$ be  positive integers.
Let $n$ be sufficiently large (as a function of $k,h$).
Let $\cF\subseteq \wb$ be a family with $|\cF|\geq (k-1+\epsilon){n\choose \fl{n/2}}$.
For each $M\in \cC(\b_n)$, let $Y(M)$ denote the set of members of $\cF$ contained in $M$.
There exist functions $X_1,\ldots, X_h$ from $\cC(\b_n)$ to $2^\cF$ such that the following holds:
\begin{enumerate}
\item $\forall M\in \cC(\b_n), X_1(M)=Y(M)$.
\item $\forall i, 1\leq i\leq h-1, \forall M\in \cC(\b_n)$  if $X_{i+1}(M)\neq \emptyset$ then $\frac{|X_{i+1}(M)|}{|X_i(M)|}\geq
1-\frac{1}{4kh}$.
\item $\forall i\in [h]$, the family of $k$-marked chains $\cL_i$ with markers in $\cF$, defined by
$$\cL_i=\left\{(M,Q): M\in \cC(\b_n), \, Q\in {X_i(M)\choose k}\right\}$$
satisfies
$$ |\cL_i|\geq (\epsilon n!/k)(1-\frac{i}{2h}).$$
\item $\forall i\in [h-1]$, every member of $\cL_{i+1}$  is good relative to  $\cL_i$ (where good and bad are defined with respect to $h$).
\end{enumerate}
\end{theorem}
\begin{proof}
We use induction on $i$. For the basis step,  for each $M\in \cC(\b_n)$, we let $X_1(M)=Y(M)$.
By Lemma \ref{many-marked-chains}, we have
\begin{equation} \label{L1}
|\cL_1|\geq (\epsilon/k) n!.
\end{equation}
So item 3 holds. There is nothing else to prove. For the induction step, let $i\geq 1$ and suppose
the functions $X_1,\ldots, X_i$ have been defined so that items 1,2,3,4 all hold. We want to define
$X_{i+1}$ to satisfy all the requirements.

For each $d\in [k]$ and each $v \in \wb$ that is $d$-lower-bad relative to $\cL_i$, we fix a corresponding $d$-lower-witness
$S_{v,d}$.  For each $d\in [k]$ and each $v\in \wb$ that is $d$-upper-bad relative to $\cL_i$, we
fix a corresponding $d$-upper-witness $T_{v,d}$.
To define $X_{i+1}$, we first classify those $M$ with $X_i(M)\neq \emptyset$ into two types. For
each $d\in [k]$, let $\b_{i,d}^-(M)$ denote
the set of vertices in $X_i(M)$ that are $d$-lower-bad relative to $M$ and $\cL_i$.
Let $\b_i^-(M)=\bigcup_{d=1}^k \b_{i,d}^-(M)$.
For each $d\in [k]$,
let $\b_{i,d}^+(M)$ denote
the set of vertices in $X_i(M)$ that are $d$-upper-bad relative to $M$ and $\cL_i$.
Let $\b_i^+(M)=\bigcup_{d=1}^k \b_{i,d}^+(M)$.
Let $\b_i(M)=\b_i^-(M)\cup \b_i^+(M)$.
Let $x(M)=|X_i(M)|$ and let   $b(M)=|\b_i(M)|$. Set $C=4kh$. Let
\begin{eqnarray*}
\cC_1&=&\{M\in \cC(\b_n): X_i(M)\neq \emptyset,\, \frac{b(M)}{x(M)}\leq \frac{1}{C}\}\\
\cC_2&=&\{M\in \cC(\b_n): X_i(M)\neq \emptyset, \,  \frac{b(M)}{x(M)}> \frac{1}{C}\}\\
\end{eqnarray*}
Now, we define $X_{i+1}$ as follows.
\begin{eqnarray*}
&&\mbox{If } X_i(M)=\emptyset \mbox{ or } M\in \cC_2, \mbox{ then let } X_{i+1}(M)=\emptyset\\
&& \mbox{Otherwise, } M\in \cC_1, \mbox{ and we let } X_{i+1}(M)=
X_i(M)\setminus \b_i(M)
\end{eqnarray*}

{\bf Claim 1.}
We have
\begin{enumerate}
\item
$\forall M\in \cC(\b_n)$, where $X_{i+1}(M)\neq \emptyset$, we have  $|X_{i+1}(M)|\geq (1-\frac{1}{C})|X_i(M)|\geq (1-\frac{1}{4kh})|X_i(M)|$.
\item
Each member of $\cL_{i+1}$ is good relative to $\cL_i$.
\item
$\sum_{M\in \cC_1}{|X_{i+1}(M)|\choose k}\geq (1-\frac{k}{C}) \sum_{M\in \cC_1}{|X_i(M)|\choose k}\geq (1-\frac{1}{4h}){|X_i(M)|\choose k}$.
\end{enumerate}
{\it Proof of Claim 1.}
Let $M\in \cC(\b_n)$ and suppose $X_{i+1}(M)\neq \emptyset$. Then $M\in \cC_1$.
By our definition of $\cC_1$, we have $|\b_i(M)|/|X_i(M)|\leq 1/C$. Since
$X_{i+1}(M)=X_i(M)\setminus \b_i(M)$,  item 1 follows immediately. The only members of $\cL_{i+1}$
have the form $(M,Q)$, where $M\in \cC_1$ and $Q\in {X_{i+1}(M)\choose k}$. Fix any such member $(M,Q)$.
Since $X_{i+1}(M)=X_i(M)\setminus \b_i(M)$, and $Q\in {X_{i+1}(M)\choose k}$,  $Q$ contains no vertex
that is either $d$-lower-bad or $d$-upper-bad relative to $M$  and $\cL_i$ for any $d\in [k]$. Hence $(M,Q)$ is good relative to $\cL_i$.   So item 2 holds.
 As in the definition,  let $b(M)=|\b_i(M)|$ and $x(M)=|X_i(M)|$. The number of $k$-subsets of $X_i(M)$ that
contain a member of $\b_i(M)$ is certainly at most
$$b{x-1\choose k-1}=\frac{bk}{x}{x\choose k}\leq \frac{k}{C}{x\choose k}.$$
Therefore, we have
$${|X_{i+1}(M)|\choose k}\geq {x\choose k}-\frac{k}{C}{x\choose k}=(1-\frac{k}{C}){|X_i(M)|\choose k}=(1-\frac{1}{4h}){|X_i(M)|\choose k}.$$
So item 3 (of Claim 1) holds. \qed

\medskip

{\bf Claim 2.}
For large $n$, we have $\sum_{M\in \cC_2} {|X_i(M)|\choose k} \leq \frac{4k}{n^{1/3}}\cdot n!$.

\medskip

{\it Proof of Claim 2.} We further partition $\cC_2$ into two subclasses. Let $\cC_2^-$ consist of
those $M\in \cC_2$ with $|\b_i^-(M)|\geq |\b_i(M)|/2=b(M)/2$ and let $\cC_2^+=\cC_2-\cC_2^-$.
For each $d\in [k]$, let $\cC_{2,d}^-$ consist of those $M\in\cC_2^-$ with $|\b_{i,d}^-(M)|\geq |\b_i^-(M)|/k$.
Clearly, $\cC_2^-=\bigcup_{d=1}^k \cC_{2,d}^-$. For each $d\in [k]$, we first bound $\sum_{M\in \cC_{2,d}^-(M)}
{|X_i(M)\choose k}$.

For each $M\in \cC_{2,d}^-$, we define a sequence $R_d^-(M)$,
called the {\it greedy $d$-lower-bad string generated by} $M$
relative to $\cL_i$, as follows. Scan $M$ from top to bottom. Let
$x_1$ be the first vertex in $\b_{i,d}^-(M)$ that we encounter.
Recall that this means  $x_1$ is $d$-lower-bad relative to $M$ and
$\cL$ and we have fixed a $d$-lower-witness  $S_{x_1,d}$ of $v$
(relative to $\cL_i$)  with $|S_{x_1,d}|\leq h$ and there is at least one
member $(M,Q)$ of $\cL_i(x_1,d)$. Since the members of $\cL_i$  on
$M$ form ${X_i(M)\choose k}$ and $\cL_i(x_1,d)\neq \emptyset$, in
particular the $k$ consecutive members of $X_i(M)$ with $x_1$
being the $d$-th one among them form a $Q$ with $(M,Q)\in
\cL_i(x_1,d)$. Since $x_1$ is  $d$-lower-bad relative to $\cL_i$,
$Q$ must intersect $D^*(x_1,S_{x_1,d})$, which takes place below
$x_1$. Let $y_1$ be the first member of $X_i(M)$ below $x_1$  that
lies in $D^*(x_1,S_{x_1,d})$. By our discussion above, $y_1$ is
among the $k-d$ members of $X_i(M)$ below $x_1$. After we
encounter $y_1$, we continue down $M$. If there are more vertices
in $X_i(M)$ that are $d$-lower-bad relative to $M$ and $\cL_i$,
then let $x_2$ denote the next vertex in $X_i(M)$ that is
$d$-lower-bad relative to $M$ and $\cL_i$. We then similarly
define $y_2$. We continue like this until we run out of vertices
in $X_i(M)$. Following our reasoning for the existence of $y_1$,
whenever an $x_i$ is defined, $y_i$ must exist and is within the
$k-d$ members of $X_i(M)$ below $x_i$. Suppose
$R_d^-(M)=(x_1,y_1,x_2,y_2,\ldots, x_p,y_p)$. By our procedure,
$p\geq \ce{|\b_{i,d}^-(M)|/k}$. Let $J$ be the increasing sequence
of $2p$ numbers in $[n]$ such that $M[J]=R_d^-(M)$. We denote $J$
by $P_d^-(M)$ and call it the {\it $d$-lower-bad profile} of $M$
relative to $\cL_i$. Now we organize the terms in $\sum_{M\in
\cC_{2,d}^-}{X_i(M)\choose k}$ by $|P_d^-(M)|$.  For convenience,
we will view the increasing sequence $P_d^-(M)$ simply as a subset
of $[n]$. Let $p$ be any positive integer. Consider $M\in
\cC_{2,d}^-$ with $|P_d^-(M)|=2p$. By item 2 of the induction
hypothesis,
$$\frac{|X_i(M)|}{|X_1(M)|}\geq (1-\frac{1}{4kh})^{i-1}\geq (1-\frac{1}{4kh})^h\geq 1-\frac{h}{4kh}>\frac{1}{2}.$$
So $|Y(M)|=|X_1(M)|\leq 2|X_i(M)|\leq 2|\b_i(M)|C\leq 4|\b_i^-(M)|C\leq 4k|\b_{i,d}^-(M)|\leq
4k^2pC$ (recall that $p\geq \frac{|\b_{i,d}^-(M)|}{k}$). Clearly the largest
number in $P_d^-(M)$ is no more than $|Y(M)|\leq 4k^2pC$.
So,  $P_d^-(M)\in {[4k^2pC]\choose 2p}$. Fix any $2p$-subset (increasing sequence) $J$ of $[4k^2pC]$.
By our definition of $P_d^-(M)$, if $P_d^-(M)=J$, then certainly $M[J]=R_d^-(M)$ forms a $d$-lower-bad string  relative to $\cL_i$ by the definition of  $R_d^-(M)$.
Thus
\begin{eqnarray*}
|\{ M\in \cC_{2,d}^-: P_d^-(M)=J\}| &\leq& |\{M\in \cC(\b_n): M[J] \mbox{ forms a $d$-lower-bad string relative to }\cL_i \}|\\
&\leq& (\frac{27h\sqrt{n\ln n}}{n})^p\cdot n!  \quad \quad \mbox{(by Lemma \ref{exponential})}
\end{eqnarray*}

So
$$|\{M\in \cC_{2,d}^-: |P_d^-(M)|=2p\}|\leq {4k^2pC\choose 2p}\cdot (\frac{27h\sqrt{n\ln n}}{n})^p\cdot n!\leq 2^{4k^2pC}(\frac{27h\sqrt{n\ln n}}{n})^p\cdot n!. $$
Also, for each $M\in \cC_{2,d}^-$ with $|P_d^-(M)|=2p$, we showed earlier that $|Y(M)|\leq 4k^2pC$. Hence
$${|X_i(M)|\choose k}\leq {|Y(M)\choose k}\leq 2^{|Y(M)|}\leq 2^{4k^2pC}.$$
So, the contribution to $\sum_{M\in \cC_{2,d}^-}{|X_i(M)|\choose k}$ from those $M\in \cC_{2,d}^-$ with $|P_d^-(M)|=2p$ is at most
$$2^{4k^2pC}\cdot 2^{4k^2pC}\cdot (\frac{27h\sqrt{n\ln n}}{n})^p\cdot n!\leq (\frac{1}{n^{1/3}})^p\cdot n!,$$
for sufficiently large $n$. Summing over all $p\geq 1$, we get
\begin{eqnarray*}
\sum_{M\in \cC^-_{2,d}}{|X_i(M)|\choose k}\leq \sum_{p=1}^\infty (\frac{1}{n^{1/3}})^p \cdot n!\leq \frac{2}{n^{1/3}} \cdot n!,
\end{eqnarray*}
for large $n$.  Summing over all $d\in [k]$, we get
\begin{eqnarray*}
\sum_{M\in \cC^-_2}{|X_i(M)|\choose k}\leq  \frac{2k}{n^{1/3}} \cdot n!,
\end{eqnarray*}

By a similar argument, we have
\begin{eqnarray*}
\sum_{M\in \cC^+_2}{|X_i(M)|\choose k}\leq \frac{2k}{n^{1/3}} \cdot n!,
\end{eqnarray*}
for large $n$. Hence
\begin{eqnarray*}
\sum_{M\in \cC_2}{|X_i(M)|\choose k}\leq \frac{4k}{n^{1/3}} \cdot n!,
\end{eqnarray*}
for large $n$. \qed

\medskip

{\bf Claim 3.} We have $$|\cL_{i+1}|\geq (\epsilon/k)n!(1-\frac{i+1}{2h}).$$

\medskip

{\it Proof of Claim 3.} By induction hypothesis,
$$|\cL_i|\geq (\epsilon/k)n!(1-\frac{i}{2h}).$$
By Claim 2, $\sum_{M\in \cC_2}{|X_i(M)|\choose k}\leq \frac{4k}{n^{1/3}}\cdot n!\leq (\epsilon/k)n!(\frac{1}{4h})$, for large $n$. So
$$\sum_{M\in \cC_1}{|X_i(M)|\choose k}\geq (\epsilon/k)n!(1-\frac{i}{2h} - \frac{1}{4h}).$$
By Claim 1 and our definition of $\cL_{i+1}$, we have
\begin{eqnarray*}
|\cL_{i+1}|&=&\sum_{M\in \cC_1}{|X_{i+1}(M)|\choose k}\geq (1-\frac{1}{4h})\sum_{M\in \cC_1} {|X_i(M)|\choose k}\\
&\geq&(\epsilon/k)n!(1-\frac{i}{2h}-\frac{1}{4h})(1-\frac{1}{4h})\\
&\geq&(\epsilon/k)n!(1-\frac{i+1}{2h})
\end{eqnarray*}
\qed

\medskip

So item 3 of the theorem holds. This completes the induction step and the proof.
\end{proof}

\section{Proof of Theorem \ref{upper2}}

Now, we are ready to prove Theorem \ref{upper2}.  We keep all the notations from previous sections.
Let $\cL_1\supseteq \cL_2\supseteq \ldots \supseteq \cL_h$ be the sequence of families of $k$-marked chains
we obtained in Theorem \ref{nested}.
 We  define a sequence of subposets $H_1, H_2,\ldots$ of $H$ as follows.
Let $H_1=H$. Recall that $H_1$ is $k$-saturated. Suppose $H_1$ is not a chain. Then by Lemma \ref{interval-removal}, $H_1$ contains a chain interval $I_1=[v_1,u_1]$ or $[u_1,v_1]$, where $v_1$ is a leaf in $D(H_1)$
and $H_2=H_1\setminus (I - u_1)$ is still $k$-saturated and $D(H_2)$ is a tree.
If $H_2$ is a chain, then we terminate. Otherwise, $H_2$ contains
a chain interval $I_2=[v_2,u_2]$ or $[u_2,v_2]$  such that $H_3=H_2\setminus (I_2-u_2)$ is $k$-saturated.
We continue like this until the current subposet, say $H_q$, is just a $k$-chain.
Clearly $q\leq h$. We prove the following proposition, which implies Theorem \ref{upper2}.
Given a set $W$ of vertices in $\b_n$, we view $W$ as a family of subsets of $[n]$ and define
 the {\it sublattice} of $\b_n$ induced by $W$, denoted by
$\b_n[W]$,  to be $(W,\subseteq)$. Clearly, $\b_n[W]$ is an induced subposet of $\b_n$.

\begin{proposition}
There exist subsets $W_1\supseteq W_2\supseteq \ldots \supseteq W_q$ of $\b_n$ such that
\begin{enumerate}
\item $\forall i\in [q]$, $\b_n[W_i]=H_i$. (Hence, we will treat $W_i$ as $V(H_i)$.)
\item  $\forall i\in [q], \forall v\in W_i=V(H_i)$ if $v$ is at level $d$ of $H_i$ (from the top) then $\cL_i(v,d)\neq \emptyset$.
\end{enumerate}
\end{proposition}
\begin{proof}
We use reverse induction on $i$. For the basis step, let $i=q$. We know that $H_q$ is just a $k$-chain.
By Theorem \ref{nested}, $|\cL_q|\geq (\epsilon/k)n!(1-\frac{q}{2h})>0$. So there exists $(M,Q)\in \cL_q$.
We embed $H_q$ using $Q$. Let $W_q=V(Q)$. Clearly, items 1 and 2 both hold.
For the induction step, let $i\leq q-1$. Suppose we have defined $W_{i+1}, \ldots, W_q$ that satisfy
all the requirements. Recall that $H_{i+1}=H_i\setminus (I_i-u_i)$, where $I_i=[v_i,u_i]$ or $[u_i,v_i]$ is a chain
interval in $H_i$. Without loss of generality, we may assume $I_i=[v_i,u_i]$, which would put $v_i$ at level $k$ since $v_i$
is a leaf in $D(H_i)$ and each leaf is at level $1$ or $k$. (The case where $I_i=[u_i,v_i]$ can be handled similarly.)
Suppose $u_i$ is
at level $d$ from the top in $H_{i+1}$. By item 2 of the induction
hypothesis, $\cL_{i+1}(u_i,d)\neq \emptyset$. Let $(M,Q)\in \cL_{i+1}(u_i,d)$. Then $u_i$
is the $d$-th vertex of $Q$ (from the top). By Theorem \ref{nested}, $(M,Q)$ is good relative to $\cL_i$.
Let $S=W_{i+1}\setminus U(u_i)$. In other words, $S$ is the
set of vertices in $H_{i+1}$ that are not ancestors of $u_i$. Since $|S|\leq h$, by Proposition \ref{disjoint},
there exists a member $(M',Q')\in \cL_i(u_i,d)$ that is disjoint from $D^*(u_i, S)$. We can embed $I_i-u_i$
using the portion $Q^*$ of $Q'$ below $u_i$. The newly embedded vertices, by design, are not in $D^*(u_i,S)$ and hence are not related to any vertex in $S$. (They are, however, descedants of $u_i$ and hence are still descendants
of the ancestors of $u_i$ in $W_{i+1}=V(H_{i+1})$.) Let $W_i=W_{i+1}\cup V(Q^*)$. Since
$\b_n[W_{i+1}]=H_{i+1}$, it follows from our discussion above that $\b_n[W_i]=H_i$.
Furthermore, because of the existence of $(M',Q')$ it is easy to see that the newly embedded vertices
(namely those in $Q^*$) still satisfy item 2 of the  theorem.
This completes the induction step and the proof.
\end{proof}


\section{Concluding remarks}

\subsection{Comments on the approach}
Even though our approach follows that of Bukh,  we needed to use several key new ideas.
In Bukh's argument, it is crucial to assume that on
each full chain the number of members of $\cF$ is bounded. Indeed, if some full chain contains
$h$ members of $\cF$ then $\cF$ contains an $h$-chain, which already contains $H$ as a subposet.  However,
for the induced version, this is no longer the case. One can have an unbounded number of  members of $\cF$ on a full chain without forcing
an induced $H$. To overcome this difficulty, we consider two types of full chains. In one type of full chains
 the number of bad members of $\cF$ is negligible compared to the number of members of $\cF$. In a second type of full
chains, the number of bad members of $\cF$ is comparable to the number of members of $\cF$. For the second type, the key
observation is that the number of $k$-marked chains on type $2$ full chains decreases exponentially fast as the number
of bad members of $\cF$ that lie on the full chain. This still allows us to limit the total number of bad $k$-marked chains
and build our nested sequence of dense families of $k$-marked chains, which is then used to embed $H$ iteratively.
Another major departure from Bukh's approach is that we no longer insist on using entire $k$-marked chains to embed maximal chains of $H$. Rather, we use $k$-marked chains to locate good vertices to embed $H$, while preserving the levels of vertices.

\subsection{Induced versus non-induced}
We showed that when $H$ is a poset whose Hasse diagram is a tree
$La(n,H)$ and $La^*(n,H)$ are asymptotically equal, both
asymptotic to $(k-1){n\choose \fl{n/2}}$, where $k$ is the height
of $D(H)$. For other posets though, $La^*(n,H)$ can be very
different from $La(n,H)$. For instance, since $La(n,K_{r,s})\leq
(2+o(1)){n\choose \fl{n/2}}$, for any two level poset $H$ we have
$La(n,H)\leq (2+o(1)){n\choose \fl{n/2}}$. However, we now show
that for every fixed $m$, there exists a two level poset $H_m$
satisfying $La^*(n,H_m)\geq (m-1-o(1)){n\choose \fl{n/2}}$.
Specifically,  let $H_m$ be the two level poset consisting of
$x_1,x_2,\ldots, x_m$ at level one and $y_1,y_2,\ldots, y_m$ at
level two. For each $i\in [m]$, let $x_i\leq y_j$ for
$j=i,i+1,\ldots, m$. Suppose $\cG\subseteq \b_n$ is a family that
contains $H_m$ as an induced subposet with members $A_1,\ldots,
A_m$ playing the roles of $x_1,\ldots, x_m$, respectively and
members  $B_1,\ldots, B_m$ playing the role of $y_1,\ldots,y_m$,
respectively.  For each $i\in [m]$, let $S_i=\bigcap_{j=i}^m B_i$.
Note that $S_m\supseteq S_{m-1}\supseteq \ldots \supseteq S_1$.
Also, by our assumption $\forall i\in [m], A_i\subseteq S_i$ and
if $i\geq 2$ then also $A_i\not\subseteq S_{i-1}$. In particular,
this implies that $S_1,\ldots,S_m$ must be distinct sets. So
$|S_m|-|S_1|\geq m-1$. It follows that $|B_m|-|A_1|\geq m-1$. Now,
let $\cF\subseteq \b_n$ be a family that consists of the middle
$m-1$ levels of $\b_n$. Since the cardinalities of any two members
of $\cF$ differ by at most $m-2$,, $\cF$ does not contain $H_m$ as
an induced subposet. Since $|\cF|=(m-1-o(1)){n\choose \fl{n/2}}$,
we have $La^*(n,H_m)\geq (m-1-o(1)){n\choose \fl{n/2}}$.


\end{document}